\documentclass[12pt]{article}
\usepackage{amsmath,amssymb,graphicx}
\newtheorem{theorem}{Theorem}[section]

\newtheorem{lemma}[theorem]{Lemma}
\def\bull{\vrule height .9ex width .8ex depth -.1ex}
\newenvironment{proof}{\smallbreak \noindent {\bf Proof.~}}
              {\unskip\nobreak\hfill\hskip 2em \bull\par\medbreak}
\newenvironment{proofof}[1]{\medbreak\noindent{\bf Proof of~#1.~}}
              {\unskip\nobreak\hfill\hskip 2em \bull\par\medbreak}
\def\bQ{\mathbb{Q}}
\def\bR{\mathbb{R}}
\def\cG{\mathcal{G}}
\def\cL{\mathcal{L}}
\def\eps{\epsilon}
\def\la{\lambda}
\def\om{\omega}
\def\si{\sigma}
\def\SAF{\mathop{\mathrm{Inv}}}
\title{Notes on the commutator group of\\ the group of interval exchange
transformations}
\author{Yaroslav Vorobets\thanks{%
        Partially supported by the NSF grant DMS-0701298.}}
\date{}
\begin{document}

\maketitle

\begin{abstract}
We study the group of interval exchange transformations and obtain several
characterizations of its commutator group.  In particular, it turns out
that the commutator group is generated by elements of order $2$.
\end{abstract}

\section{Introduction}\label{intro}

Let $(p,q)$ be an interval of the real line and $p_0=p<p_1<\dots<p_{k-1}
<p_k=q$ be a finite collection of points that split $(p,q)$ into
subintervals $(p_{i-1},p_i)$, $i=1,2,\dots,k$.  A transformation $f$ of
the interval $(p,q)$ that rearranges the subintervals by translation is
called an {\em interval exchange transformation\/} (see Figure \ref{fig1}).
To be precise, the restriction of $f$ to any $(p_{i-1},p_i)$ is a
translation, the translated subinterval remains within $(p,q)$ and does not
overlap with the other translated subintervals.  The definition is still
incomplete as values of $f$ at the points $p_i$ are not specified.  The
standard way to do this, which we adopt, is to require that $f$ be right
continuous.  That is, we consider the half-closed interval $I=[p,q)$
partitioned into smaller half-closed intervals $I_i=[p_{i-1},p_i)$.  The
interval exchange transformation $f$ is to translate each $I_i$ so that the
images $f(I_1),f(I_2),\dots,f(I_k)$ form another partition of $I$.

Let $\la$ be a $k$-dimensional vector whose coordinates are lengths of the
intervals $I_1,I_2,\dots,I_k$.  Let $\pi$ be a permutation on
$\{1,2,\dots,k\}$ that tells how the intervals are rearranged by $f$.
Namely, $\pi(i)$ is the position of $f(I_i)$ when the the intervals
$f(I_1),\dots,f(I_k)$ are ordered from left to right.  For the example in
Figure \ref{fig1}, $\pi=(1\,2\,4\,3)$.  We refer to the pair $(\la,\pi)$ as
a combinatorial description of $f$.  Given an integer $k\ge1$, a
$k$-dimensional vector $\la$ with positive coordinates that add up to the
length of $I$, and a permutation $\pi$ on $\{1,2,\dots,k\}$, the pair
$(\la,\pi)$ determines a unique interval exchange transformation of $I$.
The converse is not true.  Any partition of $I$ into subintervals that are
translated by $f$ gives rise to a distinct combinatorial description.
Clearly, such a partition is not unique.  However there is a unique
partition with the smallest possible number of subintervals.

\begin{figure}[t]
\centerline{\includegraphics[scale=1.2]{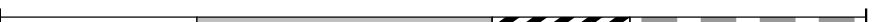}}
\caption{Interval exchange transformation.}\label{fig1}
\end{figure}

The interval exchange transformations have been popular objects of study in
ergodic theory.  First of all, the exchange of two intervals is equivalent
to a rotation of the circle (it becomes one when we identify the endpoints
of the interval $I$ thus producing a circle).  The exchanges of three or
more intervals were first considered by Katok and Stepin \cite{KS}.  The
systematic study started since the paper by Keane \cite{Keane} who coined
the term.  For an account of the results, see the survey by Viana
\cite{Viana}.

All interval exchange transformations of a fixed interval $I=[p,q)$ form a
transformation group $\cG_I$.  Changing the interval, one obtains an
isomorphic group.  Indeed, let $J=[p',q')$ be another interval and $h$ be
an affine map of $I$ onto $J$.  Then $f\in G_J$ if and only if $h^{-1}fh\in
G_I$.  We refer to any of the groups $G_I$ as the {\em group of interval
exchange transformations}.  The present notes are concerned with
group-theoretic properties of $G_I$.  An important tool here is the {\em
scissors congruence invariant\/} or the {\em Sah-Arnoux-Fathi (SAF)
invariant\/} of $f\in G_I$ introduced independently by Sah \cite{Sah} and
Arnoux and Fathi \cite{Arnoux}.  The invariant can be informally defined by
$$
\SAF(f)=\int_I 1\otimes\bigl(f(x)-x\bigr)\,dx
$$
(the integral is actually a finite sum).  The importance stems from the
fact that $\SAF$ is a homomorphism of the group $G_I$ onto
$\bR\wedge_{\bQ}\bR$.  As a consequence, the invariant vanishes on the
commutator group, which is a subgroup of $G_I$ generated by commutators
$f^{-1}g^{-1}fg$, where $f$ and $g$ run over $G_I$.

In this paper, we establish the following properties of the commutator
group of the group of interval exchange transformations $G_I$.

\begin{theorem}\label{main1}
The following four groups are the same:
\begin{itemize}
\item
the group of interval exchange transformations with zero SAF invariant,
\item
the commutator group of the group of interval exchange transformations,
\item
the group generated by interval exchange transformations of order $2$,
\item
the group generated by interval exchange transformations of finite order.
\end{itemize}
\end{theorem}

\begin{theorem}\label{main2}
The quotient of the group of interval exchange transformations by its
commutator group is isomorphic to $\bR\wedge_{\bQ}\bR$.
\end{theorem}

\begin{theorem}\label{main3}
The commutator group of the group of interval exchange transformations is
simple.
\end{theorem}

It has to be noted that most of these results are already known.  A theorem
by Sah reproduced in Veech's paper \cite{Veech} contains Theorems
\ref{main2}, \ref{main3}, and part of Theorem \ref{main1}.  Unfortunately,
the preprint of Sah \cite{Sah} was never published.  Hence we include
complete proofs.  A new result of the present paper is that the commutator
group of $\cG_I$ is generated by elements of order $2$.  This is the
central result of the paper as our proofs of the theorems are based on the
study of elements of order $2$.

The paper is organized as follows.  Section \ref{elem} contains some
elementary constructions that will be used in the proofs of the theorems.
The scissors congruence invariant is considered in Section \ref{inv}.
Section \ref{comm} is devoted to the proof of Theorem \ref{main1}.  Theorem
\ref{main2} is proved in the same section.  Section \ref{simp} is devoted
to the proof of Theorem \ref{main3}.

The author is grateful to Michael Boshernitzan for useful and inspiring
discussions.

\section{Elementary constructions}\label{elem}

Let us choose an arbitrary interval $I=[p,q)$.  In what follows, all
interval exchange transformations are assumed to be defined on $I$.  Also,
all subintervals of $I$ are assumed to be half-closed intervals of the form
$[x,y)$.

The proofs of Theorems \ref{main1} and \ref{main3} are based on several
elementary constructions described in this section.  First of all, we
introduce two basic types of transformations used in those constructions.
An {\em interval swap map of type $a$} is an interval exchange
transformation that interchanges two nonoverlapping intervals of length $a$
by translation while fixing the rest of the interval $I$.  A {\em
restricted rotation of type $(a,b)$} is an interval exchange transformation
that exchanges two neighboring intervals of lengths $a$ and $b$ (the
interval of length $a$ must be to the left of the interval of length $b$)
while fixing the rest of $I$.  The type of an interval swap map is
determined uniquely, and so is the type of a restricted rotation.  Clearly,
any interval swap map is an involution.  The inverse of a restricted
rotation of type $(a,b)$ is a restricted rotation of type $(b,a)$.  Any
restricted rotation of type $(a,a)$ is also an interval swap map of type
$a$.

\begin{figure}[t]
\centerline{\includegraphics[scale=1]{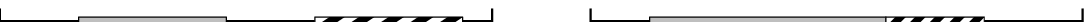}}
\caption{An interval swap map and a restricted rotation.}\label{fig2}
\end{figure}

\begin{lemma}\label{elem1}
Any interval exchange transformation $f$ is a product of restricted
rotations.  Moreover, if $f$ exchanges at least two intervals and $\cL$ is
the set of their lengths, it is enough to use restricted rotations of types
$(a,b)$ such that $a,b\in\cL$.
\end{lemma}

\begin{proof}
The exchange of one interval is the identity.  In this case, take any
restricted rotation $h$.  Then $f=hh^{-1}$, which is a product of
restricted rotations.  Now assume that $f$ exchanges $k\ge2$ intervals.
Let $I_1,I_2,\dots,I_k$ be the intervals, ordered from left to right, and
$(\la,\pi)$ be the corresponding combinatorial description of $f$.  Then
$\cL$ is the set of coordinates of the vector $\la$.

For any permutation $\si$ on $\{1,2,\dots,k\}$, let $f_\si$ denote a unique
interval exchange transformation with the combinatorial description
$(\la,\si)$.  Given two permutations $\si$ and $\tau$ on $\{1,2,\dots,k\}$,
let $g_{\si,\tau}=f_{\si}f_\tau^{-1}$.  For any $i$, $1\le i\le k$, the
transformation $g_{\si,\tau}$ translates the interval $f_\tau(I_i)$ onto
$f_\si(I_i)$.  It follows that $g_{\si,\tau}$ has combinatorial
description $(\la',\si\tau^{-1})$, where $\la'_i=\la_{\tau^{-1}(i)}$ for
$1\le i\le k$.  Now suppose that $\pi$ is expanded into a product of
permutations $\pi=\si_1\si_2\dots\si_m$.  For any $j$, $1\le j\le m$, let
$\pi_j=\si_j\si_{j+1}\dots\si_m$.  Then $f=h_1h_2\dots h_m$, where
$h_m=f_{\si_m}$ and $h_j=g_{\pi_j,\pi_{j+1}}$ for $1\le j<m$.  By the above
each $h_j$ has combinatorial description $(\la^{(j)},\si_j)$, where the
vector $\la^{(j)}$ is obtained from $\la$ by permutation of its
coordinates.  In the case $\si_j$ is a transposition of neighboring
numbers, $\si_j=(i,\,i+1)$, the transformation $h_j$ is a restricted
rotation of type $\bigl(\la^{(j)}_i,\la^{(j)}_{i+1}\bigr)$.  Notice that
$\la^{(j)}_i,\la^{(j)}_{i+1}\in\cL$.

It remains to observe that any permutation $\pi$ on $\{1,2,\dots,k\}$ is a
product of transpositions of neighboring numbers.  Indeed, we can represent
$\pi$ as a product of cycles.  Further, any cycle $(n_1\,n_2\,\dots\,n_m)$
of length $m\ge2$ is a product of $m-1$ transpositions:
$(n_1\,n_2\,\dots\,n_m)=(n_1\,n_2)(n_2\,n_3)\dots(n_{m-1}\,n_m)$.  A cycle
of length $1$ is the identity, hence it equals $(1\,2)(1\,2)$.  Finally,
any transposition $(n\,l)$, $n<l$, is expanded into a product of
transpositions of neighboring numbers: $(n\,l)=\tau_n\tau_{n+1}\dots
\tau_{l-2}\tau_{l-1}\tau_{l-2}\dots\tau_{n+1}\tau_n$, where $\tau_i
=(i,\,i+1)$.
\end{proof}

Notice that the set $\cL$ in Lemma \ref{elem1} depends on the combinatorial
description of the interval exchange transformation $f$.  The lemma holds
for every version of this set.

\begin{lemma}\label{elem2}
Any interval exchange transformation of finite order is a product of
interval swap maps.
\end{lemma}

\begin{proof}
Suppose $J_1,J_2,\dots,J_k$ are nonoverlapping intervals of the same length
$a$ contained in an interval $I$.  Let $g$ be an interval exchange
transformation of $I$ that translates $J_i$ onto $J_{i+1}$ for $1\le i<k$,
translates $J_k$ onto $J_1$, and fixes the rest of $I$.  If $k\ge2$ then
$g$ is the product of $k-1$ interval swap maps of type $a$.  Namely,
$g=h_1h_2\dots h_{k-1}$, where each $h_i$ interchanges $J_i$ with $J_{i+1}$
by translation while fixing the rest of $I$.  In the case $k=1$, $g$ is the
identity.  Then $g=h_0h_0$ for any interval swap map $h_0$ on $I$.

Let $f$ be an interval exchange transformation of $I$ that has finite
order.  Since there are only finitely many distinct powers of $f$, there
are only finitely many points in $I$ at which one of the powers has a
discontinuity.  Let $I=I_1\cup I_2\cup\ldots\cup I_m$ be a partition of $I$
into subintervals created by all such points.  By construction, the
restriction of $f$ to any $I_i$ is a translation and, moreover, the
translated interval $f(I_i)$ is contained in another element of the
partition.  Since the same applies to the inverse $f^{-1}$, it follows that
$f(I_i)$ actually coincides with some element of the partition.  Hence $f$
permutes the intervals $I_1,I_2,\dots,I_m$ by translation.  Therefore these
intervals can be relabeled as $J_{ij}$, $1\le i\le l$, $1\le j\le k_i$ ($l$
and $k_1,\dots,k_l$ are some positive integers), so that $f$ translates
each $J_{ij}$ onto $J_{i,j+1}$ if $j<k_i$ and onto $J_{i1}$ if $j=k_i$.
For any $i\in\{1,2,\dots,l\}$ let $g_i$ be an interval exchange
transformation that coincides with $f$ on the union of intervals $J_{ij}$,
$1\le j\le k_i$, and fixes the rest of $I$.  It is easy to observe that the
transformations $g_1,\dots,g_l$ commute and $f=g_1g_2\dots g_l$.  By the
above each $g_i$ can be represented as a product of interval swap maps.
Hence $f$ is a product of interval swap maps as well.
\end{proof}

\begin{lemma}\label{elem3}
Any interval swap map is a commutator of two interval exchange
transformations of order $2$.
\end{lemma}

\begin{proof}
Let $f$ be an interval swap map of type $a$.  Let $I_1=[x,x+a)$ and
$I_2=[y,y+a)$ be nonoverlapping intervals interchanged by $f$.  We split
the interval $I_1$ into two subintervals $I_{11}=[x,x+a/2)$ and
$I_{12}=[x+a/2,x+a)$.  Similarly, $I_2$ is divided into $I_{21}=[y,y+a/2)$
and $I_{22}=[y+a/2,y+a)$.  Now we introduce three interval swap maps of
type $a/2$: $g_1$ interchanges $I_{11}$ with $I_{12}$, $g_2$ interchanges
$I_{21}$ with $I_{22}$, and $g_3$ interchanges $I_{11}$ with $I_{21}$.  The
maps $f,g_1,g_2,g_3$ permute the intervals $I_{11},I_{12},I_{21},I_{22}$ by
translation and fix the rest of the interval $I$.  It is easy to see that
$g_1g_2=g_2g_1$.  Hence $g=g_1g_2$ is an element of order $2$.  Further, we
check that $g_3g=g_3g_2g_1$ maps $I_{11}$ onto $I_{12}$, $I_{12}$ onto
$I_{21}$, $I_{21}$ onto $I_{22}$, and $I_{22}$ onto $I_{11}$.  Therefore
the second iteration $(g_3g)^2$ interchanges $I_{11}$ with $I_{21}$ and
$I_{12}$ with $I_{22}$, which is exactly how $f$ acts.  Thus $f=(g_3g)^2
=g_3^{-1}g^{-1}g_3g$.
\end{proof}

For the next two constructions, we need another definition.  The {\em
support\/} of an interval exchange transformation $f$ is the set of all
points in $I$ moved by $f$.  It is the union of a finite number of
(half-closed) intervals.  For instance, the support of a restricted
rotation of type $(a,b)$ is a single interval of length $a+b$.  The support
of an interval swap map of type $a$ is the union of two nonoverlapping
intervals of length $a$.  Note that any interval swap map is uniquely
determined by its type and support.  The same holds true for any restricted
rotation.

\begin{lemma}\label{elem4}
Let $f_1$ and $f_2$ be interval swap maps of the same type.  If the
supports of $f_1$ and $f_2$ do not overlap then there exists an interval
exchange transformation $g$ of order $2$ such that $f_2=gf_1g$.
\end{lemma}

\begin{proof}
Let $a$ be the type of $f_1$ and $f_2$.  Let $I_1$ and $I'_1$ be
nonoverlapping intervals of length $a$ interchanged by $f_1$.  Let $I_2$
and $I'_2$ be nonoverlapping intervals of length $a$ interchanged by $f_2$.
Assume that the supports of $f_1$ and $f_2$ do not overlap, i.e., the
intervals $I_1,I'_1,I_2,I'_2$ do not overlap with each other.  Let us
introduce two more interval swap maps of type $a$: $g_1$ interchanges $I_1$
with $I_2$ and $g_2$ interchanges $I'_1$ with $I'_2$.  Since the supports
of $g_1$ and $g_2$ do not overlap, the transformations commute.  Hence the
product $g=g_1g_2$ is an element of order $2$.  The maps $f_1$, $f_2$, and
$g$ permute the intervals $I_1,I'_1,I_2,I'_2$ by translation and fix the
rest of the interval $I$.  One easily checks that $f_2=gf_1g$.
\end{proof}

\begin{lemma}\label{elem5}
Let $f_1$ and $f_2$ be restricted rotations of the same type.  If the
supports of $f_1$ and $f_2$ do not overlap then $f_1^{-1}f_2$ is the
product of three interval swap maps.
\end{lemma}

\begin{proof}
Let $(a,b)$ be the type of $f_1$ and $f_2$.  Let $I_1=[x,x+a+b)$ be the
support of $f_1$ and $I_2=[y,y+a+b)$ be the support of $f_2$.  The
transformation $f_2$ translates the interval $I_{21}=[y,y+a)$ by $b$ and
the interval $I_{22}=[y+a,y+a+b)$ by $-a$.  The inverse $f_1^{-1}$ is a
restricted rotation of type $(b,a)$ with the same support as $f_1$.  It
translates the interval $I_{11}=[x,x+b)$ by $a$ and the interval
$I_{12}=[x+b,x+a+b)$ by $-b$.

Assume that the supports $I_1$ and $I_2$ do not overlap.  Let $g_1$ be the
interval swap map of type $a$ that interchanges the intervals $I_{12}$ and
$I_{21}$, let $g_2$ be the interval swap map of type $b$ that interchanges
$I_{11}$ and $I_{22}$, and let $g_3$ be the interval swap map of type $a+b$
that interchanges $I_1$ and $I_2$.  It is easy to check that
$f_1^{-1}f_2=g_3g_2g_1=g_3g_1g_2$ (see Figure \ref{fig3}).
\end{proof}

\begin{lemma}\label{elem6}
Suppose $f$ is a restricted rotation of type $(a,b)$, where $a>b$.  Then
there exist interval swap maps $g_1$ and $g_2$ such that $g_1f=fg_2$ is a
restricted rotation of type $(a-b,b)$.
\end{lemma}

\begin{figure}[t]
\centerline{\includegraphics[scale=1]{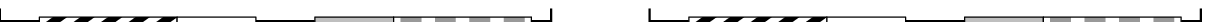}}
\caption{Proof of Lemma \ref{elem5}.}\label{fig3}
\end{figure}

\begin{proof}
Let $I_0=[x,x+a+b)$ be the support of $f$.  We define three more
transformations with supports inside $I_0$: $g_1$ is an interval swap map
of type $b$ that interchanges the intervals $[x,x+b)$ and $[x+a,x+a+b)$,
$g_2$ is an interval swap map of type $b$ that interchanges $[x+a-b,x+a)$
and $[x+a,x+a+b)$, and $h$ is a restricted rotation of type $(a-b,b)$ with
support $[x,x+a)$.  Let us check that $g_1h=f$.  Since $a>b$, the points
$x+a-b$ and $x+a$ divide $I_0$ into three subintervals $I_1=[x,x+a-b)$,
$I_2=[x+a-b,x+a)$, and $I_3=[x+a,x+a+b)$.  The map $h$ translates $I_1$ by
$b$, $I_2$ by $b-a$, and fixes $I_3$.  Then the map $g_1$ translates $I_3$
by $-a$, $[x,x+b)=h(I_2)$ by $a$, and fixes $[x+b,x+a)=h(I_1)$.  Therefore
the product $g_1h$ translates $I_1$ by $b$, $I_2$ by $b$, and $I_3$ by
$-a$.  This is exactly how $f$ acts.  Similarly, we check that $f=hg_2$.
It remains to notice that $g_1f=g_1^2h=h=hg_2^2=fg_2$.
\end{proof}

\begin{lemma}\label{elem7}
Let $f$ be a nontrivial interval exchange transformation.  Then there exist
$\eps_0>0$ and, for any $0<\eps<\eps_0$, interval swap maps $g_1,g_2$ such
that $g_2f^{-1}g_1fg_1g_2$ is an interval swap map of type $\eps$.
\end{lemma}

\begin{proof}
Since $f$ is not the identity, we can find an interval $J=[x,y)$ such that
$f$ translates $J$ by some $t\ne0$.  Let $\eps_0=\min(y-x,|t|)$.  Given any
$\eps$, $0<\eps<\eps_0$, we introduce two intervals $I_0=[x,x+\eps)$ and
$I_1=[x+t,x+t+\eps)$.  By construction, $I_0$ and $I_1$ do not overlap.
Besides, $f$ translates $I_0$ onto $I_1$.  Let $g_0$ be an interval swap
map of type $\eps/2$ that interchanges two halves $I_{01}=[x,x+\eps/2)$ and
$I_{02}=[x+\eps/2,x+\eps)$ of the interval $I_0$.  Let $g_1$ be an interval
swap map of type $\eps/2$ that interchanges two halves
$I_{11}=[x+t,x+t+\eps/2)$ and $I_{12}=[x+t+\eps/2,x+t+\eps)$ of $I_1$.
Since $f$ translates $I_0$ onto $I_1$, it follows that $g_0=f^{-1}g_1f$.
Further, let $g_2$ be an interval swap map of type $\eps/2$ that
interchanges $I_{02}$ with $I_{11}$.  The maps $g_0,g_1,g_2$ permute the
nonoverlapping intervals $I_{01},I_{02},I_{11},I_{12}$ by translation and
fix the rest of the interval $I$.  It is easy to check that $g_2g_0g_1g_2
=g_2f^{-1}g_1fg_1g_2$ is an interval swap map of type $\eps$ that
interchanges $I_0$ with $I_1$.
\end{proof}

\section{Scissors congruence invariant}\label{inv}

Let us recall the construction of the tensor product.  Suppose $V$ and $W$
are vector spaces over a field $F$.  Let $Z(V,W)$ be a vector space over
$F$ with basis $\{z[v,w]\}_{(v,w)\in V\times W}$.  Let $Y(V,W)$ denote the
subspace of $Z(V,W)$ spanned by all vectors of the form
$z[v_1+v_2,w]-z[v_1,w]-z[v_2,w]$, $z[v,w_1+w_2]-z[v,w_1]-z[v,w_2]$,
$z[\alpha v,w]-\alpha z[v,w]$, and $z[v,\alpha w]-\alpha z[v,w]$, where
$v,v_1,v_2\in V$, $w,w_1,w_2\in W$, and $\alpha\in F$.  The {\em tensor
product\/} of the spaces $V$ and $W$ over the field $F$, denoted
$V\otimes_F W$, is the quotient of the vector space $Z(V,W)$ by $Y(V,W)$.
For any $v\in V$ and $w\in W$ the coset $z[v,w]+Y(V,W)$ is denoted
$v\otimes w$.  By construction, $(v,w)\mapsto v\otimes w$ is a bilinear
mapping on $V\times W$.  In the case $V=W$, for any vectors $v,w\in V$ we
define the {\em wedge product\/} $v\wedge w=v\otimes w-w\otimes v$.  The
subspace of $V\otimes_F V$ spanned by all wedge products is denoted
$V\wedge_F V$.  By construction, $(v,w)\mapsto v\wedge w$ is a bilinear,
skew-symmetric mapping on $V\times V$.

\begin{lemma}\label{inv1}
Suppose $V$ is a vector space over a field $F$ and $v_1,v_2,\dots,v_k\in V$
are linearly independent vectors.  Then the wedge products $v_i\wedge v_j$,
$1\le i<j\le k$, are linearly independent in $V\wedge_F V$.
\end{lemma}

\begin{proof}
For any bilinear function $\om:V\times V\to F$ let $\tilde\om$ denote a
unique linear function on $Z(V,V)$ such that $\tilde\om(z[v,w])=\om(v,w)$
for all $v,w\in V$.  Since $\om$ is bilinear, the function $\tilde\om$
vanishes on the subspace $Y(V,V)$.  Hence it gives rise to a linear
function $\hat\om:V\otimes_F V\to F$.  By construction, $\hat\om(v\otimes
w)=\om(v,w)$ for all $v,w\in V$.

Let us extend the set $\{v_1,v_2,\dots,v_k\}$ to a basis $S$ for the vector
space $V$.  For any $l,m\in\{1,2,\dots,k\}$ we denote by $\om_{lm}$ a
unique bilinear function on $V\times V$ such that $\om_{lm}(v,w)=1$ if
$(v,w)=(v_l,v_m)$ and $\om_{lm}(v,w)=0$ for any other pair $(v,w)\in
S\times S$.  The function $\om_{lm}$ gives rise to a linear function
$\hat\om_{lm}$ on $V\otimes_F V$ as described above.  For any
$i,j\in\{1,2,\dots,k\}$, $i\ne j$, we have $\hat\om_{lm}(v_i\wedge v_j)=1$
if $i=l$ and $j=m$, $\hat\om_{lm}(v_i\wedge v_j)=-1$ if $i=m$ and $j=l$,
and $\hat\om_{lm}(v_i\wedge v_j)=0$ otherwise.

Consider an arbitrary linear combination
$$
\xi=\sum\nolimits_{1\le i<j\le k}r_{ij}(v_i\wedge v_j)
$$
with coefficients $r_{ij}$ from $F$.  It is easy to observe that
$\hat\om_{lm}(\xi)=r_{lm}$ for any $1\le l<m\le k$.  Therefore $\xi\ne0$
unless all $r_{ij}$ are zeros.  Thus the wedge products $v_i\wedge v_j$,
$1\le i<j\le k$, are linearly independent over $F$.
\end{proof}

Let $f$ be an interval exchange transformation of an interval $I=[p,q)$.
Consider an arbitrary partition of $I$ into subintervals, $I=I_1\cup I_2
\cup\ldots\cup I_k$, such that the restriction of $f$ to any $I_i$ is a
translation by some $t_i$.  Let $\la_i$ be the length of $I_i$, $1\le i\le
k$.  The {\em scissors congruence invariant}, also known as the {\em
Sah-Arnoux-Fathi (SAF) invariant}, of $f$ is
$$
\SAF(f)=\la_1\otimes t_1+\la_2\otimes t_2+\cdots+\la_k\otimes t_k
$$
regarded as an element of the tensor product $\bR\otimes_{\bQ}\bR$.  One
can easily check that $\SAF(f)=a\wedge b$ for any restricted rotation $f$
of type $(a,b)$ and $\SAF(g)=0$ for any interval swap map $g$.  The term
`scissors congruence invariant' is partially explained by the following
lemma.

\begin{lemma}\label{inv2}
The scissors congruence invariant $\SAF(f)$ of an interval exchange
transformation $f$ does not depend on the combinatorial description of $f$.
\end{lemma}

\begin{proof}
Let $I=I_1\cup\ldots\cup I_k$ be a partition of the interval $I$ into
subintervals such that the restriction of $f$ to any $I_i$ is a translation
by some $t_i$.  Let $I=I'_1\cup\ldots\cup I'_m$ be another partition into
subintervals such that the restriction of $f$ to any $I'_j$ is a
translation by some $t'_j$.  Let $\la_i$ denote the length of $I_i$ ($1\le
i\le k$) and $\la'_j$ denote the length of $I'_j$ ($1\le j\le m$).  We have
to show that $\xi=\la_1\otimes t_1+\cdots+\la_k\otimes t_k$ coincides with
$\xi'=\la'_1\otimes t'_1+\cdots+\la'_m\otimes t'_m$ in
$\bR\otimes_{\bQ}\bR$.

For any $1\le i\le k$ and $1\le j\le m$ the intersection $I_i\cap I'_j$ is
either an interval or the empty set.  We let $\mu_{ij}$ be the length of
the interval in the former case and $\mu_{ij}=0$ otherwise.  Further, let
$$
\eta=\sum\nolimits_{i=1}^k\sum\nolimits_{j=1}^m \mu_{ij}\otimes t_i, \qquad
\eta'=\sum\nolimits_{i=1}^k\sum\nolimits_{j=1}^m \mu_{ij}\otimes t'_j.
$$
Clearly, $t_i=t'_j$ whenever $I_i\cap I'_j$ is an interval.  Otherwise
$\mu_{ij}=0$ and $0\otimes t_i=0=0\otimes t'_j$.  In any case,
$\mu_{ij}\otimes t_i=\mu_{ij}\otimes t'_j$.  Therefore $\eta=\eta'$.  For
any $i\in\{1,2,\dots,k\}$, nonempty intersections $I_i\cap I'_j$, $1\le
j\le m$, form a partition of the interval $I_i$ into subintervals.  Hence
$\la_i=\mu_{i1}+\mu_{i2}+\dots+\mu_{im}$.  It follows that $\eta=\xi$.
Similarly, we obtain that $\eta'=\xi'$.  Thus $\xi=\eta=\eta'=\xi'$.
\end{proof}

In view of Lemma \ref{inv2}, for any interval $I=[p,q)$ we can consider the
invariant $\SAF$ as a function on $\cG_I$, the set of all interval exchange
transformations of $I$.

\begin{lemma}\label{inv3}
The scissors congruence invariant $\SAF$ is a homomorphism of the group
$\cG_I$ to $\bR\otimes_{\bQ}\bR$.
\end{lemma}

\begin{proof}
Consider arbitrary interval exchange transformations $f$ and $g$ of the
interval $I$.  We have to show that $\SAF(fg)=\SAF(f)+\SAF(g)$.  Let
$I=I_1\cup I_2\cup\ldots\cup I_k$ be a partition of $I$ into subintervals
such that the restrictions of both $g$ and $fg$ to any $I_i$ are
translations by some $t_i$ and $t'_i$, respectively.  Let $\la_i$ be the
length of $I_i$, $1\le i\le k$.  Then
\begin{eqnarray*}
\SAF(g) &=& \la_1\otimes t_1+\la_2\otimes t_2+\cdots+\la_k\otimes t_k,\\
\SAF(fg) &=& \la_1\otimes t'_1+\la_2\otimes t'_2+\cdots+\la_k\otimes t'_k.
\end{eqnarray*}
It is easy to see that for any $1\le i\le k$ the image $g(I_i)$ is an
interval of length $\la_i$ and the restriction of $f$ to $g(I_i)$ is the
translation by $t'_i-t_i$.  Besides, the intervals $g(I_1),g(I_2),\dots,
g(I_k)$ form another partition of $I$.  It follows that
$$
\SAF(f)=\la_1\otimes(t'_1-t_1)+\la_2\otimes(t'_2-t_2)+\cdots
+\la_k\otimes(t'_k-t_k).
$$
Since $\la_i\otimes(t'_i-t_i)+\la_i\otimes t_i=\la_i\otimes t'_i$ for all
$1\le i\le k$, we obtain that $\SAF(fg)=\SAF(f)+\SAF(g)$.
\end{proof}

In the remainder of this section we show that $\SAF$ is actually a
homomorphism of $\cG_I$ onto $\bR\wedge_{\bQ}\bR$.

\begin{lemma}\label{inv4}
For any $a,b,\eps>0$ there exist pairs of positive numbers
$(a_1,b_1)$, $(a_2,b_2),\dots,(a_n,b_n)$ such that
\begin{itemize}
\item
$(a_1,b_1)=(a,b)$,
\item
$(a_{i+1},b_{i+1})=(a_i-b_i,b_i)$ or $(a_{i+1},b_{i+1})=(a_i,b_i-a_i)$ for
$1\le i\le n-1$,
\item
$a_n+b_n<\eps$ or $a_n=b_n$.
\end{itemize}
\end{lemma}

\begin{proof}
We define a finite or infinite sequence of pairs inductively.  First of
all, $(a_1,b_1)=(a,b)$.  Further, assume that the pair $(a_i,b_i)$ is
defined for some positive integer $i$.  If $a_i=b_i$ then this is the last
pair in the sequence.  Otherwise we let $(a_{i+1},b_{i+1})=(a_i-b_i,b_i)$
if $a_i>b_i$ and $(a_{i+1},b_{i+1})=(a_i,b_i-a_i)$ if $a_i<b_i$.  Since
$a,b>0$, it follows by induction that $a_i,b_i>0$ for all $i$.  If the
sequence $(a_1,b_1),(a_2,b_2),\dots$ is finite and contains $n$ pairs, then
$a_n=b_n$ and we are done.  If the sequence is infinite, it is enough to
show that $a_n+b_n\to0$ as $n\to\infty$.  For any positive integer $n$ let
$c_n=\min(a_n,b_n)$.  Since $a_1,a_2,\dots$ and $b_1,b_2,\dots$ are
nonincreasing sequences of positive numbers, so is the sequence
$c_1,c_2,\dots$.  By construction, $a_{i+1}+b_{i+1}=(a_i+b_i)-c_i$ for all
$i$.  It follows that the series $c_1+c_2+\cdots$ is convergent.  In
particular, $c_n\to0$ as $n\to\infty$.  Note that if $c_{i+1}<c_i$ for some
$i$, then $c_i=\max(a_{i+1},b_{i+1})$ so that $a_{i+1}+b_{i+1}
=c_i+c_{i+1}$.  This implies $a_n+b_n\to0$ as $n\to\infty$.
\end{proof}

\begin{lemma}\label{inv5}
For any $a,b\in\bR$ and $\eps>0$ there exist $a_0,b_0>0$, $a_0+b_0<\eps$,
such that $a\wedge b=a_0\wedge b_0$ in $\bR\wedge_{\bQ}\bR$.
\end{lemma}

\begin{proof}
Note that $c\wedge c=0$ for all $c\in\bR$.  Therefore in the case $a\wedge
b=0$ it is enough to take $a_0=b_0=c$, where $0<c<\eps/2$.

Now assume that $a\wedge b\ne0$.  Clearly, in this case $a$ and $b$ are
nonzero.  Since $(-a)\wedge(-b)=a\wedge b$ and $(-a)\wedge b=a\wedge(-b)
=b\wedge a$ for all $a,b\in\bR$, it is no loss to assume that $a$ and $b$
are positive.  By Lemma \ref{inv4}, there exist pairs of positive numbers
$(a_1,b_1)=(a,b)$, $(a_2,b_2),\dots,(a_n,b_n)$ such that
$(a_{i+1},b_{i+1})=(a_i-b_i,b_i)$ or $(a_{i+1},b_{i+1})=(a_i,b_i-a_i)$ for
$1\le i\le n-1$, and also $a_n+b_n<\eps$ or $a_n=b_n$.  Since
$(a'-b')\wedge b'=a'\wedge b'-b'\wedge b'=a'\wedge b'$ and
$a'\wedge(b'-a')=a'\wedge b'-a'\wedge a'=a'\wedge b'$ for all
$a',b'\in\bR$, it follows by induction that $a_i\wedge b_i=a\wedge b$,
$i=1,2,\dots,n$.  Then $a_n\ne b_n$ as $a_n\wedge b_n=a\wedge b\ne0$.  Thus
$a_n+b_n<\eps$.
\end{proof}

\begin{lemma}\label{inv6}
An element $\xi\in\bR\otimes_{\bQ}\bR$ is the scissors congruence invariant
of some interval exchange transformation in $\cG_I$ if and only if
$\xi\in\bR\wedge_{\bQ}\bR$.
\end{lemma}

\begin{proof}
As already mentioned before, the SAF invariant of a restricted rotation of
type $(a,b)$ is $a\wedge b$.  By Lemma \ref{elem1}, any $f\in\cG_I$ is a
product of restricted rotations.  Since $\SAF$ is a homomorphism of the
group $\cG_I$ due to Lemma \ref{inv3}, we obtain that $\SAF(f)$ is a finite
sum of wedge products.  Hence $\SAF(f)\in\bR\wedge_{\bQ}\bR$.

Let $l$ denote the length of the interval $I$.  By Lemma \ref{inv5}, for
any $a,b\in\bR$ one can find $a_0,b_0>0$, $a_0+b_0<l$, such that $a\wedge b
=a_0\wedge b_0$.  By the choice of $a_0$ and $b_0$, the group $\cG_I$
contains a restricted rotation of type $(a_0,b_0)$.  It follows that any
wedge product in $\bR\wedge_{\bQ}\bR$ is the SAF invariant of some interval
exchange transformation in $\cG_I$.  Since $\SAF$ is a homomorphism of
$\cG_I$, any sum of wedge products is also the SAF invariant of some
$f\in\cG_I$.

Any $\xi\in\bR\wedge_{\bQ}\bR$ is a linear combination of wedge products
with rational coefficients.  Since $r(a\wedge b)=(ra)\wedge b$ for all
$a,b\in\bR$ and $r\in\bQ$, the element $\xi$ can also be represented as a
sum of wedge products.  By the above, $\xi=\SAF(f)$ for some $f\in\cG_I$.
\end{proof}

\section{Commutator group}\label{comm}

We begin this section with a technical lemma that will be used in the proof
of the principal Lemma \ref{comm2} below.

\begin{lemma}\label{comm1}
Suppose $L_1,L_2,\dots,L_k$ are positive numbers.  Then there exist
positive numbers $l_1,l_2,\dots,l_m$ linearly independent over $\bQ$ such
that each $L_i$ is a linear combination of $l_1,l_2,\dots,l_m$ with
nonnegative integer coefficients.
\end{lemma}

\begin{proof}
The proof is by induction on the number $k$ of the reals
$L_1,L_2,\dots,L_k$.  The case $k=1$ is trivial.  Now assume that $k>1$ and
the lemma holds for the numbers $L_1,L_2,\dots,L_{k-1}$.  That is, there
exist positive numbers $l_1,l_2,\dots,l_m$ linearly independent over $\bQ$
such that each $L_i$, $1\le i<k$ is a linear combination of
$l_1,l_2,\dots,l_m$ with nonnegative integer coefficients.  If the reals
$l_1,\dots,l_m$ and $L_k$ are linearly independent over $\bQ$, then we are
done.  Otherwise $L_k$ is a linear combination of $l_1,\dots,l_m$ with
rational coefficients.  Let us separate positive and negative terms in this
linear combination: $L_k=a_1l_{i_1}+\cdots+a_sl_{i_s}-(b_1l_{j_1}+\cdots+
b_pl_{j_p})$, where $a_{i_t},b_{j_t}$ are positive rationals and the
indices $i_1,\dots,i_s,j_1,\dots,j_p$ are all distinct.  It is possible
that there is no negative term at all.  Since $l_1,\dots,l_m$ and $L_k$ are
positive numbers, we can find positive rationals $r_1,\dots,r_s$ such that
$r_1+\cdots+r_s=1$ and $l'_{i_t}=a_tl_{i_t}-r_t(b_1l_{j_1}+\cdots+
b_pl_{j_p})$ is positive for $1\le t\le s$.  Let $l'_i=l_i$ for any
$1\le i\le m$ different from $i_1,\dots,i_s$.  Then $l'_1,\dots,l'_m$ are
positive numbers linearly independent over $\bQ$.  By construction,
$L_k=l'_{i_1}+\cdots+l'_{i_s}$ and $l_{i_t}=a_t^{-1}l'_{i_t}
+a_t^{-1}r_t(b_1l'_{j_1}+\cdots+b_pl'_{j_p})$ for $1\le t\le s$.  Therefore
each of the numbers $l_1,\dots,l_m$ and $L_k$ is a linear combination of
$l'_1,\dots,l'_m$ with nonnegative rational coefficients.  It follows that
each of the numbers $L_1,L_2,\dots,L_k$ is also a linear combination of
$l'_1,\dots,l'_m$ with nonnegative rational coefficients.  Then there
exists a positive integer $N$ such that each $L_i$ is a linear combination
of $l'_1/N,\dots,l'_m/N$ with nonnegative integer coefficients.
\end{proof}

Let us call a product of restricted rotations {\em balanced\/} if for any
$a,b>0$ the number of factors of type $(a,b)$ in this product matches the
number of factors of type $(b,a)$.

\begin{lemma}\label{comm2}
Any interval exchange transformation with zero SAF invariant can be
represented as a balanced product of restricted rotations.
\end{lemma}

\begin{proof}
Consider an arbitrary interval exchange transformation $f$ of an interval
$I$.  If $f$ is the identity, then for any restricted rotation $h$ on $I$
we have $f=hh^{-1}$, which is a balanced product of restricted rotations.
Now assume $f$ is not the identity.  Let $I=I_1\cup\ldots\cup I_k$ be a
partition of $I$ into subintervals such that the restriction of $f$ to any
$I_i$ is a translation.  Note that $k\ge2$.  Let $L_1,L_2,\dots,L_k$ be
lengths of the intervals $I_1,I_2,\dots,I_k$.  By Lemma \ref{comm1}, one
can find positive numbers $l_1,l_2,\dots,l_m$ linearly independent over
$\bQ$ such that each $L_i$ is a linear combination of $l_1,l_2,\dots,l_m$
with nonnegative integer coefficients.  Then each $I_i$ can be partitioned
into smaller intervals with lengths in the set $\cL=\{l_1,l_2,\dots,l_m\}$.
Clearly, the restriction of $f$ to any of the smaller intervals is a
translation, hence Lemma \ref{elem1} applies here.  We obtain that $f$ can
be represented as a product of restricted rotations, $f=f_1f_2\dots f_n$,
such that the type $(a,b)$ of any factor satisfies $a,b\in\cL$.  For any
$i,j\in\{1,2,\dots,m\}$ let $s_{ij}$ denote the number of factors of type
$(l_i,l_j)$ in this product.  Then
$$
\SAF(f)=\sum\nolimits_{i=1}^n \SAF(f_n)=\sum\nolimits_{i=1}^m
\sum\nolimits_{j=1}^m s_{ij}(l_i\wedge l_j)
=\sum\nolimits_{1\le i<j\le m}(s_{ij}-s_{ji})(l_i\wedge l_j).
$$
Since the numbers $l_1,\dots,l_m$ are linearly independent over $\bQ$, it
follows from Lemma \ref{inv1} that the wedge products $l_i\wedge l_j$,
$1\le i<j\le m$, are linearly independent in $\bR\wedge_{\bQ}\bR$.
Therefore $\SAF(f)=0$ only if $s_{ij}=s_{ji}$ for all $i,j$, $i<j$.  Then
$s_{ij}=s_{ji}$ for all $i,j\in\{1,2,\dots,m\}$, which means that the
product $f_1f_2\dots f_n$ is balanced.
\end{proof}

The next lemma is an extension of Lemma \ref{elem6} that will be used in
the proofs of Lemmas \ref{comm4} and \ref{comm5} below.

\begin{lemma}\label{comm3}
Given $a,b,\eps>0$, there exist $a_0,b_0>0$, $a_0+b_0<\eps$, such that any
restricted rotation $f$ of type $(a,b)$ can be represented as $f=hg$, where
$h$ is a restricted rotation of type $(a_0,b_0)$ and $g$ is a product of
interval swap maps.
\end{lemma}

\begin{proof}
Consider an arbitrary restricted rotation $f$ of type $(a',b')$, where
$a'\ne b'$.  If $a'>b'$ then Lemma \ref{elem6} implies that $f=hg$, where
$h$ is a restricted rotation of type $(a'-b',b')$ and $g$ is an interval
swap map.  In the case $a'<b'$, we observe that the inverse map $f^{-1}$ is
a restricted rotation of type $(b',a')$.  The same Lemma \ref{elem6}
implies that $f^{-1}=\tilde g\tilde h$, where $\tilde h$ is a restricted
rotation of type $(b'-a',a')$ and $\tilde g$ is an interval swap map.  Note
that $f=\tilde h^{-1}\tilde g^{-1}=\tilde h^{-1}\tilde g$ and $\tilde
h^{-1}$ is a restricted rotation of type $(a',b'-a')$.

By Lemma \ref{inv4}, there exist pairs of positive numbers
$(a_1,b_1),\dots,(a_n,b_n)$ such that
\begin{itemize}
\item
$(a_1,b_1)=(a,b)$,
\item
$(a_{i+1},b_{i+1})=(a_i-b_i,b_i)$ or $(a_{i+1},b_{i+1})=(a_i,b_i-a_i)$ for
$1\le i\le n-1$,
\item
$a_n+b_n<\eps$ or $a_n=b_n$.
\end{itemize}
Clearly, $a_i\ne b_i$ for $1\le i<n$.  By induction, it follows from the
above that there exist interval exchange transformations
$f_1=f,f_2,\dots,f_n$ and $g_2,\dots,g_n$ such that $f_i$ is a restricted
rotation of type $(a_i,b_i)$, $g_i$ is an interval swap map, and
$f_{i-1}=f_ig_i$ for $2\le i\le n$.  We have $f=f_ng$, where $g=g_ng_{n-1}
\dots g_2$ is a product of interval swap maps.  If $a_n+b_n<\eps$ then we
are done.  Otherwise $a_n=b_n$ so that $f_n$ itself is an interval swap
map, hence $f$ is a product of interval swap maps.  In this case, take an
arbitrary restricted rotation $h$ of type $(a_0,b_0)$, where
$a_0=b_0<\eps/2$.  Since $h$ is also an interval swap map, we obtain
$f=h(hf)$, where $hf$ is a product of interval swap maps.
\end{proof}

\begin{lemma}\label{comm4}
Let $f_1$ and $f_2$ be restricted rotations of the same type.  Then
$f_1^{-1}f_2$ is a product of interval swap maps.
\end{lemma}

\begin{proof}
The lemma has already been proved in one case.  If the supports of $f_1$
and $f_2$ do not overlap then $f_1^{-1}f_2$ is the product of three
interval swap maps due to Lemma \ref{elem5}.  We are going to reduce the
general case to this particular one.

Let $(a,b)$ be the type of the restricted rotations $f_1$ and $f_2$.  First
we assume there exists an interval $I_0\subset I$ of length $a+b$ that does
not overlap with supports of $f_1$ and $f_2$.  Let $f_0$ denote a unique
restricted rotation of type $(a,b)$ with support $I_0$.  By Lemma
\ref{elem5}, both $f_1^{-1}f_0$ and $f_0^{-1}f_2$ are products of three
interval swap maps.  Hence $f_1^{-1}f_2=(f_1^{-1}f_0)(f_0^{-1}f_2)$ is the
product of six interval swap maps.

The above assumption always holds in the case when $a+b\le l/5$, where $l$
is the length of $I$.  Indeed, let us divide the interval $I$ into $5$
pieces of length $l/5$.  Then the support of $f_1$, which is an interval of
length $a+b$, overlaps with at most two pieces.  The same is true for the
support of $f_2$.  Therefore we have at least one piece with interior
disjoint from both supports.  This piece clearly contains an interval of
length $a+b$.

Now consider the general case.  It follows from Lemma \ref{comm3} that
$f_1=h_1g_1$ and $f_2=h_2g_2$, where $g_1,g_2$ are products of interval
swap maps while $h_1,h_2$ are restricted rotations of the same type
$(a_0,b_0)$ such that $a_0+b_0<l/5$.  Note that $g_1^{-1}$ and $g_2^{-1}$
are also products of interval swap maps.  By the above $h_1^{-1}h_2$ is the
product of six interval swap maps.  Then $f_1^{-1}f_2=
g_1^{-1}(h_1^{-1}h_2)g_2$ is a product of interval swap maps as well.
\end{proof}

\begin{lemma}\label{comm5}
Let $f$ be a restricted rotation and $g$ be an arbitrary interval exchange
transformation.  Then the commutator $f^{-1}g^{-1}fg$ is a product of
interval swap maps.
\end{lemma}

\begin{proof}
Let $(a,b)$ be the type of the restricted rotation $f$ and $J$ be the
support of $f$.  First assume that the restriction of the transformation
$g^{-1}$ to $J$ is a translation.  Then $g^{-1}fg$ is also a restricted
rotation of type $(a,b)$, with support $g^{-1}(J)$.  Therefore
$f^{-1}g^{-1}fg$ is a product of interval swap maps due to Lemma
\ref{comm4}.

In the general case, we choose an interval $I_0\subset I$ such that
$g^{-1}$ is a translation when restricted to $I_0$.  Let $\eps$ denote the
length of $I_0$.  According to Lemma \ref{comm3}, we have $f=f_0g_0$, where
$g_0$ is a product of interval swap maps and $f_0$ is a restricted rotation
of some type $(a_0,b_0)$ such that $a_0+b_0<\eps$.  Obviously, $g_0^{-1}$
is also a product of interval swap maps.  Since $a_0+b_0<\eps$, there
exists a restricted rotation $f_1$ of type $(a_0,b_0)$ with support
contained in $I_0$.  By the above the commutator $f_1^{-1}g^{-1}f_1g$ is a
product of swap maps.  By Lemma \ref{comm4}, $f_0^{-1}f_1$ and
$f_1^{-1}f_0$ are also products of interval swap maps.  Note that
$$
f^{-1}g^{-1}fg=g_0^{-1}f_0^{-1}g^{-1}f_0g_0g
=g_0^{-1}(f_0^{-1}f_1)(f_1^{-1}g^{-1}f_1g)g^{-1}(f_1^{-1}f_0)g_0g.
$$
Therefore $f^{-1}g^{-1}fg=g_1g^{-1}g_2g$, where $g_1$ and $g_2$ are
products of interval swap maps.  Consider an arbitrary factorization
$g_2=h_1h_2\dots h_n$ such that each $h_i$ is an interval swap map.  Then
$g^{-1}g_2g=(g^{-1}h_1g)(g^{-1}h_2g)\dots(g^{-1}h_ng)$.  Clearly, each
$g^{-1}h_ig$ is an interval exchange transformation of order $2$ and hence
a product of interval swap maps due to Lemma \ref{elem2}.  It follows that
$f^{-1}g^{-1}fg$ can also be represented as a product of interval swap
maps.
\end{proof}

\begin{lemma}\label{comm6}
Any balanced product of restricted rotations is also a product of interval
swap maps.
\end{lemma}

\begin{proof}
The proof is by strong induction on the number $n$ of factors in a balanced
product.  Let $f=f_1f_2\dots f_n$ be a balanced product of $n$ restricted
rotations and assume that the lemma holds for any balanced product of less
than $n$ factors.  Let $(a,b)$ be the type of $f_1$.  First consider the
case $a=b$.  In this case, $f_1$ is an interval swap map.  If $n=1$ then we
are done.  Otherwise $f=f_1g$, where $g=f_2\dots f_n$ is a balanced product
of $n-1$ restricted rotations.  By the inductive assumption, $g$ is a
product of interval swap maps, and so is $f$.

Now consider the case $a\ne b$.  In this case, there is also a factor $f_k$
of type $(b,a)$.  Let $g_1$ be the identity if $k=2$ and $g_1=f_2\dots
f_{k-1}$ otherwise.  Let $g_2$ be the identity if $k=n$ and
$g_2=f_{k+1}\dots f_n$ otherwise.  We have
$$
f=f_1g_1f_kg_2=(f_1f_k)(f_k^{-1}g_1f_kg_1^{-1})(g_1g_2).
$$
Since $f_1^{-1}$ is a restricted rotation of type $(b,a)$, it follows from
Lemma \ref{comm4} that $f_1f_k=(f_1^{-1})^{-1}f_k$ is a product of interval
swap maps.  Since $f_k^{-1}g_1f_kg_1^{-1}$ is the commutator of the
restricted rotation $f_k$ and the interval exchange transformation
$g_1^{-1}$, it is a product of interval swap maps due to Lemma \ref{comm5}.
If $n=2$ then $g_1g_2$ is the identity and we are done.  Otherwise we
observe that $g_1g_2$ is a balanced product of $n-2$ restricted rotations.
By the inductive assumption, $g_1g_2$ is a product of interval swap maps,
and so is $f$.
\end{proof}

\begin{proofof}{Theorem \ref{main1}}
Let $\cG=\cG_I$ be the group of interval exchange transformations of an
arbitrary interval $I=[p,q)$.  Let $\cG_0$ be the set of all elements in
$\cG$ with zero SAF invariant.  $\cG_0$ is a normal subgroup of $\cG$ as it
is the kernel of the homomorphism $\SAF$ (see Lemma \ref{inv3}).  Let
$\cG_1$ denote the commutator group of $\cG$, i.e., the subgroup of $\cG$
generated by commutators $f^{-1}g^{-1}fg$, where $f,g\in\cG$.  Also, let
$\cG_2$ be the subgroup of $\cG$ generated by all elements of order $2$ and
$\cG_3$ be the subgroup generated by all elements of finite order.  We have
to prove that the groups $\cG_0$, $\cG_1$, $\cG_2$, and $\cG_3$ coincide.

Since the scissors congruence invariant $\SAF$ is a homomorphism of $\cG$
to an abelian group, it vanishes on every commutator.  It follows that
$\cG_1\subset\cG_0$.  Lemmas \ref{comm2} and \ref{comm6} imply that any
element of $\cG_0$ is a product of interval swap maps, which are elements
of order $2$.  Therefore $\cG_0\subset\cG_2$.  The inclusion $\cG_2\subset
\cG_3$ is trivial.  By Lemma \ref{elem2}, any element of $\cG_3$ is a
product of interval swap maps, which are commutators due to Lemma
\ref{elem3}.  Hence $\cG_3\subset\cG_1$.  We conclude that $\cG_0=\cG_1
=\cG_2=\cG_3$.
\end{proofof}

\begin{proofof}{Theorem \ref{main2}}
According to Lemma \ref{inv3}, the SAF invariant $\SAF$, regarded as a
function on the group $\cG_I$ of interval exchange transformations of an
interval $I$, is a homomorphism to $\bR\otimes_{\bQ}\bR$.  Therefore the
quotient of $\cG_I$ by the kernel of this homomorphism is isomorphic to its
image.  By Lemma \ref{inv6}, the image of the homomorphism is
$\bR\wedge_{\bQ}\bR$.  By Theorem \ref{main1}, the kernel is the commutator
group of $\cG_I$.
\end{proofof}

\section{Simplicity}\label{simp}

Let $\cG=\cG_I$ be the group of interval exchange transformations of an
arbitrary interval $I=[p,q)$.  In this section we show that the commutator
group $[\cG,\cG]$ of $\cG$ is simple.

\begin{lemma}\label{simp1}
For any $\eps>0$ the commutator group of $\cG$ is generated by interval
swap maps of types less than $\eps$.
\end{lemma}

\begin{proof}
Let $f$ be an arbitrary interval swap map in $\cG$.  Denote by $a$ the type
of $f$.  Let $[x,x+a)$ and $[y,y+a)$ be the nonoverlapping intervals
interchanged by $f$.  We choose a sufficiently large positive integer $N$
such that $a/N<\eps$.  For any $i\in\{1,2,\dots,N\}$ let $f_i$ denote the
interval exchange transformation that interchanges intervals
$[x+(i-1)a/N,x+ia/N)$ and $[y+(i-1)a/N,y+ia/N)$ by translation while fixing
the rest of the interval $I$.  It is easy to see that $f=f_1f_2\dots f_N$.
Note that each $f_i$ is an interval swap map of type $a/N<\eps$.

Let $H_\eps$ be the subgroup of $\cG$ generated by all interval swap maps
of types less than $\eps$.  By the above the group $H_\eps$ contains all
interval swap maps in $\cG$.  In view of Lemma \ref{elem2}, $H_\eps$
coincides with the subgroup of $\cG$ generated by all elements of finite
order.  By Theorem \ref{main1}, $H_\eps=[\cG,\cG]$.
\end{proof}

\begin{lemma}\label{simp2}
There exists $\eps>0$ such that any two interval swap maps in $\cG$ of the
same type $a<\eps$ are conjugated in $[\cG,\cG]$.
\end{lemma}

\begin{proof}
Let $l$ be the length of the interval $I$.  Consider arbitrary interval
swap maps $f_1,f_2\in\cG$ of the same type $a<l/10$.  Let us divide the
interval $I$ into $10$ pieces of length $l/10$.  The support of $f_1$ is
the union of two intervals of length $a$.  Since $a<l/10$, each interval of
length $a$ overlaps with at most two of the ten pieces.  Hence the support
of $f_1$ overlaps with at most $4$ pieces.  The same is true for the
support of $f_2$.  Therefore we have at least two pieces with interior
disjoint from both supports.  Clearly, one can find two nonoverlapping
intervals $I_1$ and $I_2$ of length $a$ in these pieces.  Let $f_0$ be the
interval swap map of type $a$ that interchanges $I_1$ and $I_2$ by
translation and fixes the rest of $I$.  By construction, the support of
$f_0$ does not overlap with the supports of $f_1$ and $f_2$.  It follows
from Lemma \ref{elem4} that $f_1=g_1f_0g_1$ and $f_0=g_2f_2g_2$ for some
elements $g_1,g_2\in\cG$ of order $2$.  By Theorem \ref{main1}, the
commutator group $[\cG,\cG]$ contains all elements of order $2$ in $\cG$.
In particular, it contains $f_1$, $f_2$, $g_1$, and $g_2$.  Then
$g_2g_1\in[\cG,\cG]$ as well.  Since $f_1=g_1(g_2f_2g_2)g_1
=(g_2g_1)^{-1}f_2(g_2g_1)$, the elements $f_1$ and $f_2$ are conjugated in
$[\cG,\cG]$.
\end{proof}

\begin{proofof}{Theorem \ref{main3}}
Suppose $H$ is a nontrivial normal subgroup of $[\cG,\cG]$.  Let $f$ be an
arbitrary element of $H$ different from the identity.  By Lemma
\ref{elem7}, there exist $\eps_1>0$ and, for any $0<\eps<\eps_1$, interval
swap maps $g_1,g_2\in\cG$ such that $g_2f^{-1}g_1fg_1g_2$ is an interval
swap map of type $\eps$.  The interval swap maps $g_1$ and $g_2$ are
involutions.  They belong to $[\cG,\cG]$ due to Lemma \ref{elem3}.  Since
$H$ is a normal subgroup of $[\cG,\cG]$ that contains $f$, it also contains
the interval exchange transformations $f^{-1}$, $g_1^{-1}fg_1=g_1fg_1$,
$f^{-1}g_1fg_1$, and $g_2^{-1}(f^{-1}g_1fg_1)g_2=g_2f^{-1}g_1fg_1g_2$.  We
obtain that for any $0<\eps<\eps_1$ the subgroup $H$ contains an interval
swap map of type $\eps$.  By Lemma \ref{simp2}, there exists $\eps_2>0$
such that any two interval swap maps in $\cG$ of the same type
$\eps<\eps_2$ are conjugated in $[\cG,\cG]$.  It follows that all interval
swap maps in $\cG$ of types less than $\min(\eps_1,\eps_2)$ are also in
$H$.  According to Lemma \ref{simp1}, the commutator group of $\cG$ is
generated by these maps.  Hence $H=[\cG,\cG]$.  Thus the only nontrivial
normal subgroup of $[\cG,\cG]$ is $[\cG,\cG]$ itself.  That is, $[\cG,\cG]$
is a simple group.
\end{proofof}

\end{document}